\documentclass[a4paper,12pt]{amsart}
\usepackage{amsmath,amsthm,amssymb}

\makeatletter
\@namedef{subjclassname@2020}{\textup{2020} Mathematics Subject Classification}
\makeatother

\allowdisplaybreaks[1]

\newcommand{\ity}{\infty}
\newcommand{\C}{\mathbb{C}}

\newcommand{\N}{\mathbb{N}}

\numberwithin{equation}{section}
\newtheorem{theorem}{Theorem}[section]

\theoremstyle{remark}
\newtheorem{remark}[theorem]{Remark}
\newtheorem{example}[theorem]{Example}

\thanks {The research work of first author is supported by SMS grant of BIT Mesra, Ranchi, India. The research work of second  author is supported by research fellowship from National Board for Higher Mathematics (NBHM), Mumbai, India}

\begin{document}
%\title[bungee set of entire function]{dynamics on bungee set of transcendental entire functions }
\title [bungee set, escaping set and filled julia set of entire function]{dynamics on bungee set, escaping set and filled julia set of composite  transcendental entire functions }

\author[D. Kumar]{Dinesh Kumar}
\address{Department of Mathematics, Birla Institute of Technology Mesra
Ranchi--835 215, India}

\email{dineshkumar@bitmesra.ac.in }

\author[R. Kaur]{Ramanpreet Kaur}
\address{Department of Mathematics, University of Delhi,
Delhi--110 007, India}

\email{preetmaan444@gmail.com}

%\author[S. Kumar]{Sanjay Kumar}
%
%\address{Department of Mathematics, Deen Dayal Upadhyaya College, University of Delhi,
%New Delhi--110 078, India }
%
%\email{skpant@ddu.du.ac.in }

\begin{abstract}
In this paper, we have explored some of the  properties of the Bungee set of a transcendental entire function. We have provided a class of permutable entire functions for which their Bungee sets are equal. Moreover, we have obtained a result on permutable entire functions for which the escaping set of the composite entire function equals the union of the escaping sets of the two functions.  In addition, we provide an important relation between the Bungee set of composite entire function with the Bungee set of  individual functions.
\end{abstract}

\keywords{bungee set, filled Julia set, escaping set, completely invariant, asymptotic value}

\subjclass[2020]{37F10, 30D05}

\maketitle

\section{introduction}\label{sec1}

Let $f^n$ denote the $n$-th iterate of a transcendental entire function. In complex dynamics, we usually study about the Fatou set denoted by $F(f)$ (where the dynamics is stable)  and the Julia set denoted by $J(f)$ (where the dynamics is unstable). The Fatou set and the Julia set partitions the complex plane into two disjoint sets.  For an introduction to the basic properties of these sets one can refer to \cite{berg1}. We can also partition the complex plane by considering the nature of the orbit of a point. Here, by considering the nature of the orbit of a point, we mean that we consider three different set of points which can be defined as follows:
\begin{enumerate}
\item  Escaping set (denoted by $I(f),$ contains all those points whose orbits escape to infinity) which was first introduced by Eremenko \cite{e1};
\item  Filled Julia set (contains all those points whose orbit is bounded, denoted by $K(f))$;
\item  Bungee set (contains all those points whose orbit contains at least two subsequences such that one subsequence is bounded and other escapes to infinity), denoted by $BU(f)$. 
\end{enumerate}
From the above definition of sets, we can say that $BU(f)= \C \setminus I(f)\cup K(f).$ The set $K(f)$ has been extensively studied  for  a non-linear polynomial $f$ but has not been explored much when $f$ is a transcendental entire function.  Some of the   topological properties of the set $K(f)$ where $f$ is a transcendental entire function has been discussed in \cite{osb1} (see the references therein for more information on $K(f)$). \\
Bungee set of a polynomial of degree at least $2$ is not of much interest since it turns out that if $f$ is a polynomial of degree at least $2$ then $BU(f)=\emptyset.$ However, for a rational map $R$ of degree at least $2$ which is not a polynomial, $BU(R)$ can be non-empty. For instance, consider $R(z)=1/z^d, d\geq 2.$ It can be easily seen that $BU(R)=\{z\in\C: |z|<1\cup |z|>1\}.$
The basic properties of the Bungee set of a transcendental entire function $f$ are listed as follows \cite{osb2}:
\begin{enumerate}
\item [(i)] $BU(f)\not=\emptyset;$
\item [(ii)] $BU(f)\cap J(f)\not=\emptyset.$
\end{enumerate}
It was also observed in \cite{e2} that 
$BU(f)\cap F(f)\not=\emptyset,$ for some  transcendental entire function $f$ (though notion of $BU(f)$ was not formally defined at that time).
The proof of first two properties follows from the fact that there always exists an element in $J(f)$ whose orbit is dense in $J(f)$ \cite{dom1}. 
%The example of transcendental entire function $f$ such that $BU(f)\cap F(f)\not=\emptyset$ can be find in the paper `\emph{Pathological dynamics of an entire function}' by Eremenko and Lyubich.

Osborne and Sixsmith formally defined the notion of Bungee set of a transcendental entire function \cite{osb2}. They proved the following theorem which gives a connection between a Fatou component and the Bungee set:
\begin{theorem}\label{Thm}\cite{osb2}
Let $f$ be a transcendental entire funtion such that $U\cap BU(f)\not=\emptyset$, where $U\subset F(f)$. Then 
\begin{enumerate}
\item [(a)] $U\subset BU(f)$ and U is a wandering domain of $f$;
\item [(b)] $J(f)=\partial BU(f)$.
\end{enumerate}
\end{theorem} 

%To the best of our knowledge, not much has been explored regarding Bungee set of a transcendental entire function. 
One of our principal aim is to discuss some of the basic properties satisfied by $BU(f),$ where $f$ is transcendental entire. After studying the properties of $BU(f)$, it is natural to look for the relation between Bungee set of  the composition of transcendental entire functions with the Bungee set of individual transcendental entire function. Moreover, these kind of relations are also important if one wants to talk about the Bungee set of a finitely generated semigroup of transcendental entire functions. (It is to be  mentioned that Hinkkanen and Martin \cite{martin} did the seminal work for semigroups of rational functions where they extended  the dynamics of a rational function of one complex variable to an arbitrary semigroup of rational functions. The second author got motivated from \cite{martin} and did some work on semigroups of transcendental entire functions, see \cite{ {dinesh1}, {dinesh3}, {dinesh4}}). Recall that, two entire functions $f$ and $g$ are called permutable (commuting) if $f\circ g=g\circ f$. We have provided a class of permutable entire functions for which Bungee sets are equal. Moreover, we have given a class of permutable entire functions for which the escaping set of the composite entire function equals the union of the escaping sets of the two functions. In addition, we provide an important relation between Bungee set of composite entire function with those of the individual functions. Throughout the paper, by an entire function we would mean transcendental entire, unless otherwise stated.

\section{basic properties and results on bungee set}
Recall that, a set $W$ is forward invariant under a function $g$ if $g(W)\subset W$  and $W$ is backward invariant under $g$ if $g^{-1}(W)=\{w\in\C:g(w)\in W\}\subset W.$ The set $W$ is called completely invariant under $g$ if it is both forward and backward invariant under $g.$
Also, the dynamics of an entire function $f$ is to a large extent controlled by the presence of singular values. They are of two types: critical value and asymptotic value. A critical value of $f$ is the image of a critical point, while $z_0\in\C$ is called an asymptotic value of $f$ if there is a curve $\gamma\to\ity$ such that $f(z)\to z_0$ as $z\to\ity$ along $\gamma.$ The union of critical values, asymptotic values and their finite limit points constitute the set of singular values of $f$ which is usually denoted by Sing($f^{-1}$) \cite{morosowa}.
 In \cite{osb2}, it was observed that $BU(f)$ is a completely invariant set. Also, $BU(f) = BU(f^n), n\in\N.$
These are simple consequences of  the definition of $BU(f)$
and the corresponding properties for $I(f)$ and $K(f).$ 
%----------------------------------------------------------------------

%------------------------------------------------------------------
  For a transcendental entire function $f,$ Baker and Dominguez \cite{dom1} considered the following subset of $J(f):$
%a dense $G_{\delta}$ subset of $J(f)$ in the relative topology of $J(f)$ defined as:
%\begin{lemma}\label{lem1}
% $D(f)=\{z\in J(f): \mbox{O(z) is dense in J(f)}\}$, where $\mbox{O(z)}$ denotes  orbit of point $z.$  
\[ D(f)=\{z\in J(f): O(z)\mbox{ is dense in } J(f), \mbox{where } O(z) \text{ denotes  orbit of point } z\}. 
\]
%\end{lemma}
%\begin{proof} 
%Suppose that $D(f) = \emptyset$, then $\overline{D(f)}= \emptyset$ and hence $(\overline{D(f)})^c= J(f)$, i.e., $(D(f)^c)^0=J(f)$, where ($A^{0}$ denotes interior of a set $A$ and $A^{c}$ denotes complement of a set $A$). It shows that $J(f)$ has non-empty interior which is a contradiction as $J(f)\not=\C$.
%\end{proof}
 It is easy to see that $D(f)$ is  a dense $G_{\delta}$ subset of $J(f)$ in the relative topology of $J(f)$. This, in particular, establishes that $BU(f)\neq\emptyset$ and an infinite set which is unbounded \cite{dom1}. 

We now provide a class of permutable transcendental entire functions for which corresponding Bungee sets are equal.
\begin{theorem}\label{thm2}
Let $f$ and $g$ be two permutable entire functions such that $g(z)=af(z)+b$, where $a\not=0, |a|<1$, and the growth of  sequence $\{f^n\}$ is sufficiently large than that of sequence $\{a^n\}$. Then $BU(f)=BU(g)$.
\end{theorem}
\begin{proof}
 We prove this result by showing that $BU(f)\subset BU(g)$
and vice-versa. For this, suppose that $z_0\in BU(f).$ Then there exists two subsequences $\{m_k\}, \{n_k\}$ and a constant $R>0$ such that $f^ {m_k}(z_0)\to \infty$ as $k\to\infty$ and $|f^{n_k}(z_0)|<R $ for all $k=1,2,3,\ldots$ As $g(z)= af(z)+b$, on taking $p(z)=az+b$ we get that $g(z)=p(f(z))$. Using induction, we get $g^n(z)=p^n(f^n(z)), n\in\mathbb{N}$. Now, consider\\
$\begin{array}{lll}
|g^{n_k}(z_0)|&=&|p^{n_k}(f^{n_k}(z_0))|\\
            &=&|a^{n_k}f^{n_k}(z_0)+b(1+a+\cdots+a^{n_k-1})|\\
            &\leq&|a^{n_k}||f^{n_k}(z_0)|+|b|(1+|a|+\cdots+|a|^{n_k-1})\\
            &\leq& rR+s=T \,\mbox{(say)}\\
          %  &=T (\text{say}\, T=rR+s)
\end{array}$\\
where $|a^{n_k}|\leq r (say)$ and $ |b|(1+|a|+\cdots+|a|^{n_k-1})\leq\frac{|b|}{1-|a|}= s(say)$ as $k\to\infty$.
Also,\\
$\begin{array}{lll}
|g^{m_k}(z_0)|&=&|p^{m_k}(f^{m_k}(z_0))|\\
            &=&|a^{m_k}f^{m_k}(z_0)+b(1+a+\cdots+a^{m_k-1})|\\
            &\geq& |a^{m_k}||f^{m_k}(z_0)|-|b|(1+|a|+\cdots+|a|^{m_k-1})\\
            &\geq& |a^{m_k}||f^{m_k}(z_0)|-|b|\frac{1}{1-|a|}

\end{array}$\\
Using the hypothesis, growth of the sequence $\{f^{m_k}\}$ is sufficiently large than that of sequence $\{a^{m_k}\}$ and hence $|g^{m_k}(z_0)|\to \infty$ as $k\to\infty.$
% since 
%$ |a^{m_k}|\to r_0(say)$ and $|b|(1+|a|+\cdots+|a|^{m_k-1})\to s_0(say)$ as $k\to\infty $\\
Hence $z_0\in BU(g)$, i.e., $BU(f)\subset BU(g)$. On  similar lines, one can show that $BU(g)\subset BU(f)$ and this proves the result.
\end{proof} 

Recently, it was proved by  Singh \cite{aps},  that for any bounded Fatou component $P$ of a transcendental entire function $f$, we have $\partial P\cap BU(f)=\emptyset$. In the same paper, he raised the following question.\\ 
\textbf{Problem}: If $U$ is an unbounded Fatou component of a transcendental entire function $f$ then can we still say that  $\partial P\cap BU(f)=\emptyset$ ?\\
 We found that, answer to this question is in negation  in the case when $U$ is a completely invariant Fatou component. We provide a brief justification for above question.\\
\textbf{Justification}: As $U$ is an unbounded completely invariant component so $\partial U =J(f)$ \cite{baker2'}. Also, $BU(f)\cap J(f)\not = \emptyset$. Hence, these relations together implies that $\partial U \cap BU(f)\not =\emptyset.$ \\

To illustrate the above result we consider following example.
\begin{example}\label{eg3}
Consider one parameter family of exponential maps,  $f(z) = \lambda e^z , \mbox{ for }\lambda\in \left(0,\frac{1}{e}\right).$ We know that Fatou set of $f,  F(f)$ consists of a simply connected completely invariant component $P$ (which is an attracting basin and the dynamics is attracted to the attracting fixed point contained inside $P$, see \cite{dev}) and Julia set of $f, J(f)$ is boundary of $P$. Since $P$ is completely invariant, hence it is unbounded \cite{baker2'}. As intersection of $J(f)$ and  $BU(f)$ is always non-empty, we get  $J(f)=\partial P \cap BU(f)$  is non-empty.
\end{example}

Also, Singh proved the following theorem in the same paper which says that, if $f,g$ are two transcendental entire functions then  $z_0\in BU(f\circ g)$ implies that $g(z_0)\in BU(g\circ f).$\\\
In fact, we show that the converse of this theorem also holds, viz.:
\begin{theorem}\label{thm3'}  
Suppose $f,g$ are two transcendental entire functions. Then $g(z_0)\in BU(g\circ f)$ implies that $z_0\in BU(f\circ g)$. 
\end{theorem}
\begin{proof}
Suppose that $z_0\not\in BU(f\circ g)$ which means 
that either $z_0\in I(f\circ g) $ or $z_0\in K(f\circ g).$ Firstly, suppose that $z_0\in I(f\circ g) $, i.e., $(f\circ g)^{n}(z_0)\to\infty$ as $n\to\infty$ which further implies that $(f(g\circ f)^{n-1}(g(z_0)))\to\infty$ as $n\to\infty$. This shows that $g(z_0)\in I(g\circ f)$ and hence a contradiction. Now, suppose that $z_0\in K(f\circ g)$, i.e., there exists a constant $R>0$ such that $|(f\circ g)^n(z_0)|\leq R \mbox{ for all } n\in\N$. This inequality again leads us to $|(g\circ f)^n (g(z_0))|\leq g(R)=R_0 \mbox{ for all } n\in\N$, i.e., $g(z_0)\in K(g\circ f)$ which is a contradiction. On combining both the observations, we have  $z_0\in BU(f\circ g)$. 
\end{proof}
%------------------------------------------------------
%\section{results on composite entire functions of filled Julia set and escaping set}
\section{filled julia set and escaping set of composite entire functions}
To start with, we consider   permutable transcendental entire functions $f$ and $g$ and look for relation of $K(f\circ g)$ with $K(f)$ and $K(g).$ We are also interested in finding relation of $I(f\circ g)$ with $I(f)$ and $I(g).$ It has been already proved that for permutable entire functions $f$ and $g$,  $I(f\circ g)\subset I(f)\cup I(g)$ \cite{d1}. For $K(f\circ g)$, one can think of relations like $K(f)\cup K(g)\subset K(f\circ g)$ or $K(f)\cap K(g)\subset K(f\circ g)$. We show by an example that the first relation mentioned above might not hold in full generality.
\begin{example}\label{eg1}
Consider $f(z)=z+\sin z ,\,\, g(z)=z+\sin z +2\pi.$ Notice that $f(0)=0$  which implies $0\in K(f)$. Also $g^n(0)=2n\pi, n\in \N$, i.e., orbit of $0$ escapes to infinity under $g$  and hence lies in $I(g)$. Now, $(f\circ g)^n(0)= 2n\pi, n\in\N$ which implies that $0\not\in K(f\circ g)$.
\end{example}
%From the above example, we can say that the relation $K(f)\cup K(g)\subset K(f\circ g)$ need not  hold always. 
%-------------------------------------------------------------------------
In order to explore filled Julia set of a transcendental semigroup $G,$ we need to have some results connecting filled Julia set of composite entire functions with that of individual functions. In this direction, we now establish the following result:
%-------------------------------------------------------------------------------
\begin{theorem}\label{prop1}
Let $f$ and $g$ be permutable transcendental entire functions. Then
$K(f)\cap K(g)\subset K(f\circ g)$.
\end{theorem}
\begin{proof}
Suppose that $z_0\in K(f)\cap K(g)$. Consider the array of  following sequences:\\
$\begin{array}{ccccc}
f(g(z_0)) &f^2(g(z_0))&f^3(g(z_0))\cdots\\
f(g^2(z_0))&f^2(g^2(z_0))&f^3(g^2(z_0))\cdots\\
\vdots&\vdots&\vdots\\
f(g^k(z_0))&f^2(g^k(z_0))&f^3(g^k(z_0))\cdots\\
\vdots&\vdots&\vdots

\end{array}$\\
Now, the first row of the array will converge to say, $l$ for, if it does not, then $g(z_0)\in I(f)$ which implies that $z_0\in I(f)$, a contradiction. Using continuity of $g$, we can say that second row will converge to $g(l)$ and so on. Let $X$ be the set containing all these limit points. The following are possible cases for set $X$.\\
Case 1: $l$ is periodic point of $g$. In particular, suppose that $l$ is a fixed point of $g$. Hence every sequence of the array will be bounded which shows that $z_0\in K(f\circ g)$.\\
Case 2: $X$ is infinite and bounded. Then, by Weierstrass theorem, $X$ has a limit point which also shows that $z_0\in K(f\circ g)$.\\
Case 3: $X$ is unbounded, so that there exists a subsequence which  escapes to infinity. This in turn implies that $z_0\in I(f\circ g)\subset I(f) \cup I(g)$ (see \cite{d1}) which leads to  a contradiction.\\
From all the above cases, we obtain that  $K(f)\cap K(g)\subset K(f\circ g).$
\end{proof}

%----------------------------------
The second author \cite{dinesh3} investigated the notion of escaping set $I(G)$ of a transcendental semigroup $G.$ In order to formulate results on $I(G),$ it is important to have some results connecting escaping set of composite entire functions with that of individual functions.
For two permutable  transcendental entire functions  $f$ and $g,$ it is not necessarily true that $I(f)$ is invariant under $g$ and vice-versa. However,  the next result establishes that for two  permutable transcendental entire functions  $f$ and $g,$ their respective escaping sets are invariant under both $f$ as well as $g$ if the underlying functions do not possess finite asymptotic values.

\begin{theorem}\label{thm3}
Let $f,g$ be permutable transcendental entire functions such that $f$ and $g$ do not have any finite asymptotic values. Then, $f(I(g)\subset I(g)$.
\end{theorem}
\begin{proof}  We  prove this result by contradiction. Suppose that $z_0\in f(I(g))$ but $z_0\not\in I(g)$ i.e, either $z_0\in K(g)  \mbox{ or } z_0\in BU(g)$. First, suppose that $z_0\in K(g)$ which means that there exists $R>0$ such that $|g^n(z_0)|\leq R \mbox{, for all } n\in \N$. Since, $z_0\in f(I(g))$  therefore, there exists $w_0\in I(g)$ such that $z_0=f(w_0)$. Using this, we can say that  
 $|f(g^n(w_0))|\leq R \mbox{  , for all } n\in \N$ which shows that $f$ has a finite asymptotic value and hence gives a contradiction. Now suppose that $z_0\in BU(g)$. By definition of $BU(g)$, there exists two subsequences $\{n_k\}, \{m_k\} \mbox{ and a constant } A>0 \mbox{ such that } g^{n_k}(z_0)\to\infty \mbox{ as } k\to\infty \mbox{ and } |g^{m_k}(z_0)|<A \mbox{ for all } k=1,2,3,\ldots$ This implies that $|f(g^{m_k}(w_0))|<A \mbox{ for all } k=1,2,3,\ldots$ which further implies that $w_0$ is a finite asymptotic value of $f.$ This is a contradiction and hence the result.
\end{proof}
\begin{remark}\label{rem1}
As as consequence of Theorem \ref{thm3} one obtains, $f^n((I(g)))\subset I(g),$\, $n\in\N.$
\end{remark}
\begin{remark}\label{rem2}
It was shown in \cite{d1} that for two permutable entire functions $f$ and $g,$ $I(g)\subset f(I(g))$ which further implies $I(g)\subset f^n((I(g)))$ for all $n\in\N.$ Combining this with above result, we conclude that $f^n((I(g)))=I(g)$  for all $n\in\N.$
\end{remark}
Analogously, one obtains $g(I(f))\subset I(f)$ and hence $g^n(I(f))\subset I(f)$ for all $n\in\N.$ This, in particular, establishes that for two  permutable transcendental entire functions $f$ and $g,$  $I(f)$ and $I(g)$ are completely invariant under both $f$ and $g.$  We illustrate Theorem \ref{thm3} with an example:
 
\begin{example}\label{eg2}
 Let  $f(z)=z+1+e^{-z},\, g(z)=f(z)+2\pi\iota$. Using above result, we get that $f(I(g)\subset I(g)$ as $f$ and $g$ do not have any finite asymptotic values. However, we also verify this analytically. To this end, suppose that $z_0\in f(I(g)$, i.e., there exists $w_0\in I(g)$ such that $z_0=f(w_0)$. Using induction, we obtain that
 \begin{equation}\notag
\begin{split}
 |g^{n-1}(z_0)|
&=|f^{n}(w_0)+n  2\pi\iota-2\pi\iota|\\
 &=|g^n(w_0)-2\pi\iota|\\
 &\geq|g^n(w_0)|-|2\pi\iota|
 \end{split}
\end{equation}
  which tends to infinity as $n$ tends to infinity as $w_0\in I(g).$ 
\end{example}
%---------------------------------------
Under the hypothesis of Theorem \ref{thm3}, we now establish that the escaping set of composition of two permutable entire functions  contains the union of escaping set of the two functions.
\begin{theorem}\label{thm4'}
Let  $f$ and $g$ be permutable transcendental entire functions such that $f$ and $g$ do not have any finite asymptotic values. Then $I(f)\cup I(g)\subset I(f\circ g).$
\end{theorem}
\begin{proof}

Suppose  $z_0\in I(f)\cup I(g)$. This implies that $z_0\in I(f)$ or $z_0\in I(g).$ Firstly, consider the case when $z_0\in I(f).$ Consider the array of  following sequences:\\
$\begin{array}{ccccc}
f(g(z_0)) &f^2(g(z_0))&f^3(g(z_0))\cdots\\
f(g^2(z_0))&f^2(g^2(z_0))&f^3(g^2(z_0))\cdots\\
\vdots&\vdots&\vdots\\
f(g^k(z_0))&f^2(g^k(z_0))&f^3(g^k(z_0))\cdots\\
\vdots&\vdots&\vdots

\end{array}$\\

Using Theorem \ref{thm3}, we have $g^n(I(f))\subset I(f)$ for all $n\in\N$ and so $g^k(z_0)\in I(f)$ for all $k\in\N.$ This implies that each of the rows escapes to infinity and as a result the diagonal sequence $\{(f\circ g)^n(z_0): n\in\N\}$ also escapes to infinity. This establishes that $z_0\in I(f\circ g).$
\\
In the second case when $z_0\in I(g),$ we interchange the role of $f$ and $g$ in above array of sequence, and using the fact that $f^n(I(g))\subset I(g),$ we arrive at the conclusion and hence the result.
\end{proof}
\begin{remark}\label{rem3}
It was shown in \cite{d1} that for two permutable entire functions $f$ and $g,\, I(f \circ g)\subset I(f)\cup I(g).$ This, together with above result, provides a class of permutable entire functions $f$ and $g$ that do not have any finite asymptotic values  and which satisfies $I(f \circ g)=I(f)\cup I(g).$ 
\end{remark}
%---------------------------------------------------

Using Theorem \ref{thm3} and \ref{thm4'}, we now give an important relation between Bungee set of composite entire function with those of the individual functions. This relation will be useful if one wants to study Bungee set of a transcendental semigroup.

\begin{theorem}\label{sec2, thm1}
Let $f$ and $g$ be two permutable entire functions such that both $f$ and $g$ do not have any finite asymptotic values. Then $BU(f\circ g)\subset BU(f)\cap BU(g)\subset BU(f)\cup BU(g)$.
\end{theorem}
\begin{proof}
% We can prove this by the following two methods:\\
%(a) Using Theorem \ref{thm4'}, we have $I(f)\cup I(g)\subset I(f\circ g)$ which together with Remark \ref{rem3}  implies that $I(f)\cup I(g)= I(f\circ g)$. Also, from Theorem \ref{prop1}, we have $K(f)\cap K(g)\subset K(f\circ g)$. From these two set containments, we obtain  that $K(f\circ g)^c\subset K(f)^c\cap K(g)^c \mbox{ and } I(f\circ g)^c\subset I(f)^c\cap I(g)^c$. As  $BU(f\circ g)= I(f\circ g)^c\cap K(f\circ g)^c$, we obtain that $BU(f\circ g)\subset BU(f)\cap BU(g)\subset BU(f)\cup BU(g)$.\\
 Suppose that $z_0\in BU(f\circ g).$ This implies there exists two subsequences $\{n_k\},\{m_k\}$ and a constant $R>0$ such that $(f\circ g)^{n_k}(z_0)\to\infty$ as $k\to\infty$ and $|(f\circ g)^{m_k}(z_0)|<R \mbox{ for all } k\in \mathbb{N}$. It is enough to show that $z_0\in BU(g)$ (as the proof of $z_0\in BU(f)$ follows on similar lines). For this, we will show that $z_0$ cannot be in $I(g)$ as well as $K(g).$ We divide the proof into several cases:\\
Case(1): Suppose that $g^{n_k}(z_0)\to\infty$ and $g^{m_k}(z_0)\to\infty$ as $k\to\infty$.\\
From here, we can assume that $g^n(z_0)\to\infty$ as $n\to\infty$
(for, if there exists a subsequence $\{p_k\}$ such 
that $g^{p_k}(z_0)$ stays bounded as $k\to\infty$ then 
this will imply that $z_0\in BU(g)$). Also, using Theorem \ref{thm3}, $f^k(I(g)\subset I(g)$ for every  $k\in\mathbb{N}$. In particular, $f^{m_k}(z_0)\in I(g)$ which implies that $(g\circ f)^{m_k}(z_0)\to\infty $ as $k\to\infty$ which is a  contradiction. \\
Case(2): Suppose that $g^{n_k}(z_0)\to a$ and $g^{m_k}(z_0)\to\infty$ as $k\to\infty$. This again shows that $z_0\in BU(g)$ and the result follows.\\
Case(3): Suppose that $g^{n_k}(z_0)\to a$ and $g^{m_k}(z_0)\to b$ as $k\to\infty.$   It follows that $z_0\in K(g)$ (for,  if there exists a subsequence $\{p_k\}$ for 
which $g^{p_k}(z_0)\to\infty$ as $k\to\infty$ then we are back to Case(2) 
 and hence $z_0\in BU(g)$). It can be easily observed that $f(K(g)\subset K(g)$. In particular, $f^{n_k}(z_0)\in K(g)$, i.e., there exists $A>0$ such that $|g^n(f^{n_k}(z_0))|<A \mbox{ for all } n\in\mathbb{N}$. This implies that $|(f\circ g)^{n_k}(z_0)|<A \mbox{ for all } k\in\mathbb{N}$ which is again a contradiction. Hence, from the above cases, we conclude that $z_0\in BU(g).$  The proof of $z_0\in BU(f)$ follows on similar lines. \\
Thus, we obtain that $BU(f\circ g)\subset BU(f)\cap BU(g)\subset BU(f)\cup BU(g)$.
\end{proof}
  
  Bergweiler and Wang \cite{berg4}, showed that for two entire functions $f$ and $g,$  $z\in F(f\circ g)$ if and only if $f(z)\in F(g\circ f).$
We now prove a  similar result which gives a relation between $K(f\circ g)$ and $K(g\circ f)$. 
\begin{theorem}\label{thm3''}  
Let $f,g $ be two transcendental entire functions. Then  $z_0\in K(f\circ g)$ if and only if  $ g(z_0)\in K(g\circ f)$.
\end{theorem}
\begin{proof}
 Let $z_0\in K(f\circ g).$ Then there exists a constant $A>0$ such that $|(f\circ g)^n(z_0)|\leq A \mbox{ for all } n\in\N$. Suppose that  $g(z_0)\not\in K(g\circ f).$ Then, either $g(z_0)\in BU(g\circ f)$ or $ g(z_0)\in I(g\circ f)$.\\
Case1: Let $g(z_0)\in BU(g\circ f).$ Then there exists two subsequences $\{n_k\}, \{m_k\}$ and a constant $R>0$ such that $(g\circ f)^{n_k}(g(z_0))\to\infty$ as $k\to\infty$ and $|(g\circ f)^{m_k}(g(z_0))|<R $ for all $k=1,2,3,\ldots$ In particular, for any $M$ there exists $k_0\in\N$ such that $|(g(f\circ g)^{n_k}(z_0)|>M$ for all $k\geq k_0$ which shows that $z_0\not\in K(f\circ g)$ which is a contradiction.\\
Case2: Now let $g(z_0)\in I(g\circ f)$, i.e., $(g\circ f)^n(g(z_0))\to\infty$ as $n\to\infty$ which in turn implies that $z_0\in I(f\circ g)$ which is again  a contradiction.
On combining both the above cases, we get that $ g(z_0)\in K(f\circ g)$.\\
Conversely, suppose that $ g(z_0)\in K(f\circ g)$. We show that $z_0\in K(f\circ g)$. We prove this again by contradiction. Suppose that $z_0\in BU(f\circ g).$ Therefore, there exists two subsequences $\{n_k\}, \{m_k\}$ and a constant $R>0$ such that $(f\circ g)^{n_k}(z_0)\to\infty$ as $k\to\infty$ and $|(f\circ g)^{m_k}(z_0)|<R $ for all $k=1,2,3,\ldots$ In particular, for any $M$ there exists $k_0\in\N$ such that $|(f(g\circ f)^{{n_k}-1}(g(z_0))|>M$ for all $k\geq k_0$ which shows that $g(z_0)\not\in K(g\circ f)$ which is a contradiction. Now, suppose that $ z_0\in I(f\circ g)$, i.e., $(f\circ g)^n(z_0)\to\infty$ as $n\to\infty$ which in turn implies that $g(z_0)\in I(g\circ f)$ which is  a contradiction. In both the situations, we arrive at a contradiction. Hence, we get that $z_0\in K(f\circ g)$ which completes the proof.
\end{proof}

\end{document}